\documentclass[letterpaper, 10 pt, conference]{ieeeconf}  % Comment this line out if you need a4paper

\IEEEoverridecommandlockouts                              % This command is only needed if 
\usepackage{graphicx}
\usepackage{graphics} 
\usepackage{subcaption}
\usepackage[noadjust]{cite}

\usepackage{amsmath,amssymb,lipsum}  % assumes amsmath package installed

\newtheorem{defi}{Definition}[section]
\newtheorem{lem}{Lemma}[section]
\newtheorem{prop}{Proposition}[section]
\newtheorem{thm}{Theorem}[section]

\DeclareMathOperator*{\argmin}{arg\,min}

\title{\LARGE \bf
Risk-Aware Control of Discrete-Time Stochastic Systems: \\
 Integrating Kalman Filter and Worst-case CVaR \\ in Control Barrier Functions}

\author{Masako Kishida, \emph{Senior Member, IEEE} % <-this % stops a space
\thanks{*This work was supported by JST, PRESTO Grant Number JPMJPR22C3, Japan.}% <-this % stops a space
\thanks{Masako Kishida is with the National Institute of Informatics,
        Tokyo 101-8430, Japan
        {\tt\small kishida@nii.ac.jp}}
}

\begin{document}
\maketitle
\thispagestyle{empty}
\pagestyle{empty}

%%%%%%%%%%%%%%%%%%%%%%%%%%%%%%%%%%%%%%%%%%%%%%%%%%%%%%%%%%%%%%%%%%%%%%%%%%%%%%%%
\begin{abstract}
This paper proposes control approaches for discrete-time linear systems subject to stochastic disturbances. 
It employs Kalman filter to estimate the mean and covariance of the state propagation, and the worst-case conditional value-at-risk (CVaR) to quantify the tail risk using the estimated mean and covariance. 
The quantified risk is then integrated into a control barrier function (CBF) to derive constraints for controller synthesis, addressing tail risks near safe set boundaries. Two optimization-based control methods are presented using the obtained constraints for half-space and ellipsoidal safe sets, respectively. The effectiveness of the obtained results is demonstrated using numerical simulations. 
\end{abstract}

%%%%%%%%%%%%%%%%%%%%%%%%%%%%%%%%%%%%%%%%%%%%%%%%%%%%%%%%%%%%%%%%%%%%%%%%%%%%%%%%
\section{Introduction}

% general intro
Safety-critical systems such as autonomous vehicles, aerospace vehicles, and medical devices demand high reliability due to the severe consequences of their failure or malfunction, such as loss of life, significant property damage, or environmental harm. As automation becomes more prevalent in these applications, the need for risk-aware controller design is becoming increasingly essential.

% talk about CBF in the literature and place this work among them
In this context, Control Barrier Functions (CBFs)  \cite{WieA07} are now recognized for their pivotal role. 
CBFs provide conditions for control inputs to ensure system states remain within a given safe set, thereby guaranteeing the satisfaction of safety requirements in various fields \cite{AgrS17, AmeGT14, BreP22, SeoLB22}. 

Recent advancements in CBF approaches, such as those discussed  \cite{Cla21, WanMS21, CosCT23}, have incorporated external stochastic disturbances into their considerations. 
Moreover, some approaches consider risk through chance constraints \cite{LuoSK20}, or bounds on the probability of a collision \cite{YagMF21}. 
In particular, studies \cite{AhmXA22, Kis23c, SinAA23} provide ways of considering the safety by using the quantified tail risk to avoid severe consequences.  
Additionally, although many existing approaches assume that the all the system states can be observed, which may be impractical in some real applications, 
some \cite{Cla19, DasM22, WanX22, AgrP23} deal with the cases where not all the system states can be directly accessible. 
Yet, comprehensive solutions addressing both stochastic disturbances and unmeasured states remain limited  \cite{Cla19}, \cite{YagFY21}.

% talk about main contributions
The objective of this paper is to introduce a risk-aware control approach tailored for discrete-time linear systems affected by stochastic disturbances, where not all system states are directly observable. The approach integrates Kalman filter \cite{Kal60,WelB95} with the worst-case conditional value-at-risk (CVaR) \cite{ZhuF09,ZymKR13-b} and synthesizes it with a CBF framework. This integration is highly synergistic: Kalman filter provides estimates of mean and covariance of the state propagation, essential parameters for the worst-case CVaR to quantify tail risk. This combination effectively enhances controller design using CBFs.

% talk about structure
The paper is structured as follows: Section \ref{sec:prelim} introduces the notation, definitions and fundamental results. 
Sections \ref{sec:prep} - \ref{sec:syn} are the main part of this paper: After introducing building blocks in Section \ref{sec:prep},  sets of admissible inputs are characterized using Kalman filter along with the worst-case CVaR and the control barrier function in Section \ref{sec:char},  which is followed by the controller design in Section \ref{sec:syn}.
After numerical examples in Section \ref{sec:ex}, the paper is concluded with Section \ref{sec:conc}.

%%%%%%%%%%%%%%%%%%%%%%%%%%%%%%%%%%%%%%%%%%%%%%%%%%%%%%%%%%%%%%%%%%%%%%%%%%%%%%%%
\section{Mathematical Preliminaries} \label{sec:prelim}
\subsection{Notation}
The sets of real numbers, real vectors of length $n$, real matrices of size $n \times m$, real symmetric matrices of size $n$, and positive definite matrices of size $n$ are denoted by $\mathbb{R}$, $\mathbb{R}^n$,  $\mathbb{R}^{n\times m}$, $\mathbb{S}^{n}$ and $\mathbb{S}^{n}_+$, respectively. 
For $M\in \mathbb{R}^{n\times n}$, $M \succ 0$ indicates $M$ is positive definite, $M \succcurlyeq 0$ indicates $M$ is positive semidefinite, and $\text{Tr}(M)$ denotes its the trace. 
For $M\in \mathbb{R}^{n\times m}$, $M^\top$ denotes its transpose.
For a vector $v\in \mathbb{R}^n$, $\|v\|$ denotes its Euclidean norm,  $v \geq 0$ its element-wise non-negativity, and $|v|$ the element-wise absolute value.

%------------------
\subsection{Conditional Value-at-Risk}
To present risk quantification in our approach, we begin by showing the (worst-case) CVaR.

Consider a random vector $\xi \in \mathbb{R}^n$ under the true distribution $\mathbb{P}$, characterized by its mean $\mu \in \mathbb{R}^n$ and covariance matrix $\Sigma \in \mathbb{R}^{n \times n}$. The distribution $\mathbb{P}$, representing the probability law of $\xi$, is assumed to have finite second-order moments. We define $\mathcal{P}$ as the set of all probability distributions on $\mathbb{R}^n$ with identical first- and second-order moments to $\mathbb{P}$. Formally,
\begin{align*} 
 \mathcal{P}= \left\{\mathbb{P} : \mathbb{E}_{\mathbb{P}}\left[ \begin{bmatrix}\xi_i\\ 1\end{bmatrix}\begin{bmatrix}\xi_j\\ 1\end{bmatrix}^{\!\top} \right]
=  \begin{bmatrix}  \Sigma  \delta_{ij}  & \mu \\ \mu^\top & 1 \end{bmatrix}, \forall i, j\right\},
\end{align*}
where $\delta_{ij}$ is the Kronecker delta, and $\mathbb{E}_{\mathbb{P}}[ \cdot]$ the expected value under $\mathbb{P}$. Although the exact form of $\mathbb{P}$ is unknown, it is clear that $\mathbb{P} \in \mathcal{P}$.

\begin{defi}[Conditional Value-at-Risk (CVaR)  \cite{RocU00,ZymKR13-b}]
Given a measurable loss function $L : \mathbb{R}^n \rightarrow \mathbb{R}$, a probability distribution $\mathbb{P}$ on $\mathbb{R}^n$, and a level $\varepsilon \in (0, 1)$, the CVaR at level $\varepsilon$ under $\mathbb{P}$ is defined by:
\begin{align*}
\mathbb{P}\text{-CVaR}_{\varepsilon}[L({\xi})] = \inf_{\beta \in \mathbb{R}} \left\{ \beta + \frac{1}{\varepsilon}\mathbb{E}_{\mathbb{P}}[(L({\xi})-\beta)^+]\right\}.
\end{align*}
\end{defi}
CVaR is the conditional expectation of losses exceeding the ($1-\varepsilon$)-quantile of $L$, quantifying the tail risk of loss function \cite{ZymKR13-b}.

Having introduced the definition of CVaR, we now extend this concept to its worst-case scenario, key to the proposed approach.
The worst-case CVaR  is the supremum of CVaR over a given set of probability distributions as defined below: 
\begin{defi}[Worst-case CVaR \cite{ZhuF09, ZymKR13-b}]
The worst-case CVaR over a set of distributions $\mathcal{P}$ is:
\begin{align*}
\sup_{\mathbb{P}\in \mathcal{P}}\mathbb{P}\text{-CVaR}_{\varepsilon}[L({\xi})] = \inf_{\beta \in \mathbb{R}} \left\{ \beta + \frac{1}{\varepsilon}\sup_{\mathbb{P}\in \mathcal{P}}\mathbb{E}_{\mathbb{P}}[(L({\xi})-\beta)^+]\right\}.
\end{align*}
\end{defi}
The interchangeability of the supremum and infimum is validated by the stochastic saddle point theorem \cite{ShaK02}.

The widespread adoption of (worst-case) CVaR in risk quantification is largely attributed to its coherent properties.
\begin{prop}[Coherence properties \cite{ZhuF09,Art99}]\label{prop:coh}
The worst-case CVaR is a coherent risk measure, i.e., it satisfies the following properties:
Let $L_1 = L_1(\xi)$ and $L_2 = L_2(\xi)$ be two measurable loss functions, then the followings hold.
\begin{itemize}
\item Sub-additivity: For all $L_1$ and $L_2$,
\begin{align*}
&\sup_{\mathbb{P}\in \mathcal{P}}\mathbb{P}\text{-CVaR}_{\varepsilon}[L_1+L_2]\\
 & \leq \sup_{\mathbb{P}\in \mathcal{P}}\mathbb{P}\text{-CVaR}_{\varepsilon}[L_1] +
\sup_{\mathbb{P}\in \mathcal{P}}\mathbb{P}\text{-CVaR}_{\varepsilon}[L_2]; 
\end{align*}
\item Positive homogeneity: For a positive constant $c_1>0$, 
\begin{align*}
\sup_{\mathbb{P}\in \mathcal{P}}\mathbb{P}\text{-CVaR}_{\varepsilon}[c_1L_1] =c_1 \sup_{\mathbb{P}\in \mathcal{P}}\mathbb{P}\text{-CVaR}_{\varepsilon}[L_1]; 
\end{align*}
\item Monotonicity: If $L_1\leq L_2$ almost surely, 
\begin{align*}
\sup_{\mathbb{P}\in \mathcal{P}}\mathbb{P}\text{-CVaR}_{\varepsilon}[L_1] \leq
\sup_{\mathbb{P}\in \mathcal{P}}\mathbb{P}\text{-CVaR}_{\varepsilon}[L_2];
\end{align*}
\item Translation invariance: For a constant $c_2$,
\begin{align*}
\sup_{\mathbb{P}\in \mathcal{P}}\mathbb{P}\text{-CVaR}_{\varepsilon}[L_1+c_2] =\sup_{\mathbb{P}\in \mathcal{P}}\mathbb{P}\text{-CVaR}_{\varepsilon}[L_1] +
c_2.
\end{align*}
\end{itemize}
\end{prop}

The usefulness of the worst-case CVaR in risk quantification further appears in its computational efficiency for a special, but common, case. When the loss function $L(\xi)$ is quadratic in $\xi$, the worst-case CVaR can be computed via a semidefinite program. Let the second-order moment matrix of $\xi$ 
\begin{align*}
\Omega &= \begin{bmatrix}\Sigma + \mu \mu^\top & \mu \\ \mu^\top & 1 \end{bmatrix}. 
\end{align*}
\begin{lem}[Quadratic Loss Function \cite{ZymKR13-b}, \cite{ZymKR13}] \label{lem:CVaR_quadratic}
For $L({\xi}) = \xi^\top P \xi + 2q^\top \xi + r$, where $P \in \mathbb{S}^n$, $q\in \mathbb{R}^n$, and $r\in \mathbb{R}$, the worst-case CVaR is given by:
 \begin{align*}
\sup_{\mathbb{P}\in \mathcal{P}}\mathbb{P}\text{-CVaR}_{\varepsilon}[L({\xi})] =
&\inf_{ \beta} \left\{\beta + \frac{1}{\varepsilon}\text{Tr}(\Omega N): \right.\\
& N \succcurlyeq 0, \
\left. N-\left[\begin{array}{cc}P &q \\ q^\top & r-\beta \end{array}\right] \succcurlyeq 0\right\}.
\end{align*}
\end{lem}

Similarly to the absolute value of a vector, we denote the element-wise worst-case CVaR by
\begin{align*}
\sup_{\mathbb{P}\in \mathcal{P}}\mathbb{P}\text{-CVaR}_{\varepsilon}[ v] 
=&\left[\begin{matrix}
\sup_{\mathbb{P}\in \mathcal{P}}\mathbb{P}\text{-CVaR}_{\varepsilon}[e_1^\top v] \\
\vdots\\
\sup_{\mathbb{P}\in \mathcal{P}}\mathbb{P}\text{-CVaR}_{\varepsilon}[e_n^\top v] 
\end{matrix}\right],
\end{align*}
where $e_i$ is the $i$th column of the identity matrix of size $n$.

Building on Lemma \ref{lem:CVaR_quadratic}, we now present a lemma that provides a bound on linear loss functions.
\begin{lem}[A bound on linear $L(\xi)$\cite{Kis23c}] \label{lem:CVaR_bd}
Suppose $\mu=0$. Then, for $q\in \mathbb{R}^n$, 
\begin{align}
\sup_{\mathbb{P}\in \mathcal{P}}\mathbb{P}\text{-CVaR}_{\varepsilon}[ q^\top \xi] 
\leq  |q|^\top \sup_{\mathbb{P}\in \mathcal{P}}\mathbb{P}\text{-CVaR}_{\varepsilon}[  \xi].
\end{align}
\end{lem}

%------------------
\subsection{Control Barrier Function}
Having discussed the risk quantification, we now turn our attention to CBF, an approach to ensure the system safety.
Here, we review the basic CBF with no disturbance.

Define the safe set $\mathcal{C}$ as the superlevel set
of a continuously differentiable function $h : \mathbb{R}^n \rightarrow \mathbb{R}$,
\begin{align}\label{eq:safeset}
\mathcal{C} = \{x \in  \mathbb{R}^n: h(x)\geq 0\}.
\end{align}

With $\mathcal{C}$, the CBF is defined below.
\begin{defi}[Control Barrier Function (CBF) \cite{ZenZS21}]\label{defi:CBF}
The function $h$ is a discrete-time CBF for
\begin{align}
x_{t+1} =f(x_t, u_t)
\end{align}
 on $\mathcal{C}$ if there exists an $\alpha \in [0,1)$ such that for all $x_t \in  \mathcal{C}$,
there exists a $u_t \in  \mathbb{R}^m$ such that:
\begin{align}
 h(x_{t+1} )\geq \alpha h(x_t) .
\end{align}
\end{defi}
The existence of a CBF guarantees that the control system is safe assuming that the initial state is in $\mathcal{C}$ \cite{AmeXG17}.
We will define a risk-aware CBF at the end of the next section.

%%%%%%%%%%%%%%%%%%%%%%%%%%%%%%
\section{Preparation: Integrating Kalman Filter and Worst-case CVaR
in CBF} \label{sec:prep}
This is the first section of the three main sections of this paper. 
Here, we set up the foundational elements; system model,  state estimation and CBF, for the characterization of the admissible control input and controller synthesis that follow.

\subsection{System Model}
This paper deals with the discrete-time linear stochastic system described by:
\begin{align}\begin{aligned}
x_{t+1} &=Ax_t+B u_t+w_t, \\ 
z_t &= Hx_t + v_t, \label{eq:sys}
\end{aligned}\end{align}
where $x_t \in \mathbb{R}^{n}$ is the state, $u_t \in \mathbb{R}^{m}$ is the control input, $w_t \in \mathbb{R}^{n_w}$ is the disturbance,
$z_t \in \mathbb{R}^{n_y}$ is the measurement, and $v_t \in \mathbb{R}^{n_v}$ is the noise, 
 respectively, at discrete time instant $t \in \mathbb{Z}_{\geq 0}$. 
The matrices $A \in \mathbb{R}^{n\times n}$,  $B \in \mathbb{R}^{n\times m}$, $H \in \mathbb{R}^{n\times n_y}$ are assumed to be constants.  

It is assumed that 
\begin{itemize}
\item the initial state $x_0$ is a random vector with $\mathbb{E}[x_0]=\bar{x}_{0|0}$ and covariance $\mathbb{E}[(x_0-\bar{x}_0)(x_0-\bar{x}_0)^\top]=P_{0|0}$,
\item the disturbance $w_t$ are independent and identically distributed random vectors with 
$\mathbb{E}[w_t]=0$ and finite covariance $\mathbb{E}[w_tw_t^\top]=Q$,
\item the noise $v_t$ are independent and identically distributed random vectors with 
$\mathbb{E}[v_t]=0$ and finite covariance $\mathbb{E}[v_tv_t^\top]=R$, and
\item the initial state and the noise vectors at each step $\{x_0, w_1, ..., w_t, v_1, ... ,v_t\}$ are all mutually independent.
\end{itemize}

%------------------
\subsection{Kalman Filter}\label{sec:filter}
We now introduce Kalman filter to estimate the mean and covariance of the state propagation. 

The Kalman filter, a recursive estimator, optimal for linear systems with Gaussian noise, is adapted here for disturbances and noise, which are not necessarily Gaussians.
It operates in two phases: prediction and update. In the prediction phase, the filter forecasts the next state and its uncertainty. Subsequently, during the update phase, it refines these predictions based on new measurements.

We employ standard Kalman filter notation for system \eqref{eq:sys}:
\begin{itemize}
\item \textbf{Prediction:}
  \begin{itemize}
  \item Predicted state mean at $t$: $\bar{x}_{t+1|t} = A\bar{x}_{t|t}+B u_t$
  \item Predicted error covariance: ${P}_{t+1|t} = AP_{t|t}A^\top + Q$
  \end{itemize}
\item \textbf{Update (with gain $K_{t+1}$):}
  \begin{itemize}
  \item Measurement error covariance: $S_{t+1} = HP_{t+1|t}H^\top + R$ (this is used in Section \ref{sec:Kgain})
  \item Updated state mean: $\bar{x}_{t+1|t+1} = \bar{x}_{t+1|t}+K_{t+1}(z_{t+1}-H\bar{x}_{t+1|t})$
  \item Updated error covariance: ${P}_{t+1|t+1} = (I - K_{t+1}H)P_{t+1|t}$
  \end{itemize}
\end{itemize}
How to choose the gain $K_{t+1}$ is discussed later in Section \ref{sec:Kgain}. 

We also let:
\begin{itemize}
\item $\hat{x}_{t|t}$ be the random vector with mean $\bar{x}_{t|t}$ and covariance ${P}_{t|t}$, 
\end{itemize}
and 
\begin{align}
\hat{x}_{t+1|t}  = A\hat{x}_{t|t} +Bu_t + w_t. \label{eq:x_hat}
\end{align}
be the random vector that propagates the state $\hat{x}_{t|t}$ to the next time step according to \eqref{eq:sys}.
This is crucial for evaluating the condition for the control input in the next subsection.

%---
\subsection{Risk-Aware CBF}
This subsection integrates Kalman filter and worst-case CVaR to introduce Kalman filter-based risk-aware discrete-time CBF. This function is key for defining admissible control inputs in the subsequent section.

Consider a random vector $\xi_{t|t}$ and its shifted version $\xi_{t|t}^c$, defined as:
\begin{align}
\xi_{t|t} = \left[\begin{matrix}\hat{x}_{t|t}^\top& w_t^\top \end{matrix} \right]^{\top}, \
\xi_{t|t}^c = \xi_{t|t} -\left[\begin{matrix}\bar{x}_{t|t}^\top & 0 \end{matrix} \right]^{\top}.
\end{align}
Define
\begin{align}
\mathcal{P}_{t|t}^c= \left\{\mathbb{P} : \mathbb{E}_{\mathbb{P}}\left[ \begin{bmatrix}\xi_i\\ 1\end{bmatrix}\begin{bmatrix}\xi_j\\ 1\end{bmatrix}^{\!\top} \right]
=  \begin{bmatrix} \Sigma_{\xi_{t|t}^c}  \delta_{ij}  & 0 \\ 0^\top & 1 \end{bmatrix}, \forall i, j\right\},
\end{align}
where $\Sigma_{\xi_{t|t}^c}=\left[\begin{array}{cc}{P}_{t|t}  & 0 \\ 0 & Q \end{array} \right]$.
Then, the measure $\mathbb{P}$ of $\xi_{t|t}^c$ is $\mathbb{P}\in \mathcal{P}_{t|t}^c$.

With $\mathcal{C}$ in \eqref{eq:safeset}, we define Kalman filter-based risk-aware discrete-time CBF as below.
\begin{defi}[Risk-Aware Control Barrier Function]\label{defi:CBF2}
A function $h$ is a Kalman filter-based risk-aware discrete-time CBF for system \eqref{eq:sys} on $\mathcal{C}$ if there exists an $\alpha \in [0,1)$ such that for all $\hat{x}_{t|t}$ that satisfies $\sup_{\mathbb{P}\in \mathcal{P}_{t|t}^c}\mathbb{P}\text{-CVaR}_{\varepsilon}[-  h(\hat{x}_{t|t})]\leq 0$, there exists a $u_t \in \mathbb{R}^m$ such that:
\begin{align} \label{eq:or_cbf}
\sup_{\mathbb{P}\in \mathcal{P}_{t|t}^c}\mathbb{P}\text{-CVaR}_{\varepsilon}[- h(\hat{x}_{t+1|t})+\alpha h(\hat{x}_{t|t})]\leq 0.
\end{align}
Here, $\hat{x}_{t|t}$ and $\hat{x}_{t+1|t}$ are computed by Kalman filter.
\end{defi}

As we have integrated Kalman filter into CBF, Section \ref{sec:char} will proceed to characterize sets of admissible control inputs using this, paving the way for controller synthesis.

%%%%%%%%%%%%%%%%%%%%%%%%%%%%%%%%%%%%%%%%%%%%%%%%%%%%%%%%%%%%%%%%%%%%%%%%%%%%%%%%

\section{Charactering the Set of Admissible Control Input Using Kalman Filter} \label{sec:char}
This is the second section of the three main sections. 
Here, we consider characterizing the set of admissible control input by quantifying the left-hand-side of \eqref{eq:or_cbf}
to determine the control input $u_t$ at time $t$ in the next section. Here, we still assume that the gain $K_t$ is given. 
%------------------
\subsection{Half-space safe set}\label{sec:hs}
We start with a half-space safe set scenario, where the set is defined by an affine function $h(x)$:
\begin{align}\begin{aligned} \label{eq:hs-set}
\mathcal{C}_{\text{hs}} &= \{x \in  \mathbb{R}^n: h(x)= q^\top x+r  \geq 0\}.
\end{aligned}\end{align}
This scenario allows us to reformulate the constraint \eqref{eq:or_cbf} as a linear function of the control input $u$:
\begin{thm}\label{thm:hs}
If $h(x) = q^\top x+r$, $q\in\mathbb{R}^n$ and $r\in\mathbb{R}$, then the constraint \eqref{eq:or_cbf} holds if and only if
\begin{align}
-q^\top Bu_t &\leq \phi(\hat{x}_{t|t}), \label{eq:hs}
\end{align}
where
\begin{align}\begin{aligned}
 \phi(\hat{x}_{t|t})& =-\sup_{\mathbb{P}\in \mathcal{P}_{t|t}^c}\mathbb{P}\text{-CVaR}_{\varepsilon}\left[-{q}^\top  \left[\begin{matrix}A-\alpha I& I \end{matrix}\right] \xi_{t|t}^c\right]  \\
 &\qquad +q^\top (A-\alpha I)\bar{x}_{t|t}+(1-\alpha)r.
\end{aligned}\end{align}
\end{thm}
\begin{proof}
To express the constraint in terms of the control input $u_t$, 
substitute \eqref{eq:x_hat} into the constraint \eqref{eq:or_cbf}:
\begin{align}\begin{aligned}
& \sup_{\mathbb{P}\in \mathcal{P}_{t|t}^c}\mathbb{P}\text{-CVaR}_{\varepsilon}[-q^\top (\alpha I -A)(\hat{x}_{t|t} -\bar{x}_{t|t}+\bar{x}_{t|t} +w_t)]\\
&\quad - q^\top Bu_t  -r +\alpha r \\
=& \sup_{\mathbb{P}\in \mathcal{P}_{t|t}^c}\mathbb{P}\text{-CVaR}_{\varepsilon}[-q^\top (\alpha I -A)(\hat{x}_{t|t} -\bar{x}_{t|t} + w_t)]\\
&\quad - q^\top Bu_t - q^\top (\alpha I -A) \bar{x}_{t|t} - (1-\alpha)r \leq 0.
\end{aligned}\end{align}
This is equivalent to
\begin{align}\begin{aligned}
- q^\top Bu_t
&\leq -\sup_{\mathbb{P}\in \mathcal{P}_{t|t}^c}\mathbb{P}\text{-CVaR}_{\varepsilon}\left[-{q}^\top  \left[\begin{matrix}A-\alpha I& I \end{matrix}\right] \xi_{t|t}^c\right]  \\
 & +q^\top (A-\alpha I)\bar{x}_{t|t}+(1-\alpha)r= \phi(\hat{x}_{t|t}).
\end{aligned}\end{align}
\end{proof}

%------------------------
\subsection{Ellipsoidal safe set}\label{sec:ell}
Next, we consider a scenario with an ellipsoidal safe set, defined by a positive definite matrix $E \in \mathbb{S}^n_+$ and a vector $x_c \in \mathbb{R}^n$:
\begin{align}\begin{aligned} \label{eq:ell-set}
\mathcal{C}_{\text{ell}} 
&=\{x \in  \mathbb{R}^n: h(x) = - (x-x_c)^\top E (x-x_c)+ r \geq 0 \}.
\end{aligned}\end{align}
In this case, the constraint \eqref{eq:or_cbf} can be expressed as follows:
\begin{prop}\label{prop:ell}
If $h(x) = - (x-x_c)^\top E (x-x_c)+ r$, $E \in \mathbb{S}^n_+$,  $x_c \in \mathbb{R}^n$ and $r\in\mathbb{R}$, then the constraint \eqref{eq:or_cbf} is satisfied if and only if
\begin{align}
\sup_{\mathbb{P}\in \mathcal{P}_{t|t}^c}\mathbb{P}\text{-CVaR}_{\varepsilon}\left[ ( \xi_{t|t}^c)^\top \bar{P}  \xi_{t|t}^c +2 \bar{q}^\top \xi_{t|t}^c +\bar{r} \right] \leq  0, \label{eq:ell}
\end{align}
where $\bar{P}$, $\bar{q}$, and $\bar{r}$ are defined as follows:
\begin{align}\begin{aligned}
\bar{P} & = \left[\begin{matrix}A^\top E A -\alpha E & A^\top E \\ EA & E \end{matrix}\right],\\
 \bar{q}& =  \bar{q}_1+ \bar{q}_2 ,\\ 
  \bar{q}_1& =  \left[\begin{matrix} A^\top E(A \bar{x}_{t|t} -x_c)-\alpha E(x-x_c)\\ E(A \bar{x}_{t|t} -x_c) \end{matrix}\right], \\
  \bar{q}_2 &=  \left[\begin{matrix} A^\top EBu_t\\ EBu_t \end{matrix}\right], \\
 \bar{r}& =(A \bar{x}_{t|t} +Bu_t-x_c) ^\top E(A \bar{x}_{t|t} +Bu_t-x_c) \\
 &-\alpha (\bar{x}_{t|t} -x_c)^\top E (\bar{x}_{t|t} -x_c) -(1-\alpha)r.
\end{aligned}\end{align}
\end{prop}
Although checking the satisfaction of \eqref{eq:ell} for a given $u_t$ is straightforward, characterizing the exact set of $u_t$ that satisfies \eqref{eq:ell} is not except for some simple cases. 

However, a sufficient condition for the constraint \eqref{eq:or_cbf} to be satisfied can be expressed using a quadratic function of the control input $u$ as follows:
\begin{thm}\label{thm:ell}
Let  $h(x) = - (x-x_c)^\top E (x-x_c)+ r$ with $E \in \mathbb{S}^n_+$, $x_c \in \mathbb{R}^n$, and $r\in\mathbb{R}$. Define $\bar{u}_t  = [u_t ^\top, v_t ^\top]^\top$ as a combination of the control input $u_t \in \mathbb{R}^m$ and an auxiliary variable $v_t \in \mathbb{R}^m$. Then, the constraint \eqref{eq:or_cbf} holds if:
\begin{align}\begin{aligned}  \label{eq:ell-const}
 \bar{u}_t^\top \tilde{P}\bar{u}_t + 2\tilde{q}^\top(x) \bar{u}_t+ \tilde{r}(x)\leq 0 \text{ and } \
\tilde{A}\bar{u}_t \leq 0,
\end{aligned}\end{align}
where
\begin{align}\begin{aligned} \label{eq:ell-const2}
\tilde{P} & = \left[ \begin{matrix}B^\top E B& 0\\ 0 & 0 \end{matrix}\right], \\
\tilde{q}(x) & =\left[ \begin{matrix} B^\top E( A\bar{x}_{t|t}-x_c) \\ \sup_{\mathbb{P}\in \mathcal{P}_{t|t}}\mathbb{P}\text{-CVaR}_{\varepsilon}\left[ \left[\begin{matrix} A^\top EB\\ EB \end{matrix}\right]^\top \xi_{t|t}^c  \right] \end{matrix}\right],\\
\tilde{r}(x) & =\sup_{\mathbb{P}\in \mathcal{P}_{t|t}}\mathbb{P}\text{-CVaR}_{\varepsilon}\left[  (\xi_{t|t}^c)^\top \bar{P}  \xi_{t|t}^c +2\bar{q}_1^\top  \xi_{t|t}^c  \right]  \\
&  + (A \bar{x}_{t|t} -x_c) ^\top E(A \bar{x}_{t|t} -x_c) \\
 &-\alpha (\bar{x}_{t|t} -x_c)^\top E (\bar{x}_{t|t} -x_c) -(1-\alpha),\\
\tilde{A}& = \left[ \begin{matrix}  I& -I\\ - I & -I \end{matrix}\right]. 
\end{aligned}\end{align}
\end{thm}

\begin{proof}
Applying Proposition \ref{prop:coh}:
\begin{align}\begin{aligned}
&\sup_{\mathbb{P}\in \mathcal{P}_{t|t}}\mathbb{P}\text{-CVaR}_{\varepsilon}\left[  (\xi_{t|t}^c)^\top \bar{P}  \xi_{t|t}^c +2 \bar{q}^\top \xi_{t|t}^c+\bar{r} \right] \\
\leq &\sup_{\mathbb{P}\in \mathcal{P}_{t|t}}\mathbb{P}\text{-CVaR}_{\varepsilon}\left[  (\xi_{t|t}^c)^\top \bar{P}  \xi_{t|t}^c +2\bar{q}_1^\top  \xi_{t|t}^c  \right] \\
&\qquad +2\sup_{\mathbb{P}\in \mathcal{P}_{t|t}}\mathbb{P}\text{-CVaR}_{\varepsilon}\left[ \bar{q}_2^\top  \xi_{t|t}^c  \right]  +\bar{r}.
\end{aligned}\end{align}
Further, Lemma \ref{lem:CVaR_bd} implies:
\begin{align}\begin{aligned}
&\sup_{\mathbb{P}\in \mathcal{P}_{t|t}}\mathbb{P}\text{-CVaR}_{\varepsilon}\left[ \bar{q}_2^\top  \xi_{t|t}^c  \right]  \\
\leq &\left|u_t\right|^\top \sup_{\mathbb{P}\in \mathcal{P}_{t|t}}\mathbb{P}\text{-CVaR}_{\varepsilon}\left[ \left[\begin{matrix} A^\top EB\\ EB \end{matrix}\right]^\top \xi_{t|t}^c  \right]. 
\end{aligned}\end{align}
Thus, a sufficient condition for constraint satisfaction is:
\begin{align}\begin{aligned}
&\sup_{\mathbb{P}\in \mathcal{P}_{t|t}}\mathbb{P}\text{-CVaR}_{\varepsilon}\left[  (\xi_{t|t}^c)^\top \bar{P}  \xi_{t|t}^c +2\bar{q}_1^\top  \xi_{t|t}^c  \right] \\
& +2\left|u_t\right|^\top \sup_{\mathbb{P}\in \mathcal{P}_{t|t}}\mathbb{P}\text{-CVaR}_{\varepsilon}\left[ \left[\begin{matrix} A^\top EB\\ EB \end{matrix}\right]^\top \xi_{t|t}^c  \right] +\bar{r}\leq 0.
\end{aligned}\end{align}
Introducing $v_t \geq 0$ such that $-v_t \leq u_t \leq v_t$ and completing squares leads to the conditions in \eqref{eq:ell-const}-\eqref{eq:ell-const2}.
\end{proof}

%%%%%%%%%%%%%%%%%%%%%%%%%%%%%%%%%%%%%%%%%%%%%%%%%%%%%%%%%%%%%%%%%%%%%%%%%%%%%%%%
\section{Controller Synthesis} \label{sec:syn}
This section, the final of the three main sections, presents the synthesis of controllers utilizing results from previous sections, and discusses the gain $K_{t+1}$ in our approach.

%---------------------------------------------------------------
\subsection{The gain $K_{t+1}$} \label{sec:Kgain}
This subsection consider the role of the gain $K_{t+1}$ in relation to the control input $u_t$.
Basically, the control input should satisfy the constraint \eqref{eq:or_cbf}.
By looking back Lemma \ref{lem:CVaR_quadratic}, this left-hand-side is minimized when $\Omega$, or $\Sigma$ is small.  
Consequently, we employ the gain $K_{t+1}$, known as the Kalman gain, which minimizes the error covariance ${P}_{t+1|t+1}$, which corresponds to $\Sigma$. 
The Kalman gain is given by:
\begin{align}\label{eq:Kgain}
K_{t+1}= P_{t+1|t}H^\top S^{-1}_{t+1}.
\end{align}
With this, the error covariance of updated state estimate is 
\begin{align}\label{eq:KgainP}
 {P}_{t+1|t+1}  =(I-K_{t+1}H){P}_{t+1|t}. 
\end{align}
Together with Section \ref{sec:filter}, the Kalman filter formula is completed.

%---------------------------------------------------------------
\subsection{Method 1: Modifying a nominal controller}

One approach to designing a controller using \eqref{eq:or_cbf} is to adapt a nominal controller $u_{\text{nom}}(x)$, which does not account for safety constraints, to comply with derived safety conditions by
minimally modifying it so as to ensure the safety  \cite{CosCT23},  \cite{AmeCE19}.
The control input at $t$ can be computed by 
\begin{align} \label{eq:controller}
\begin{aligned} 
u^*(x) &=  \text{argmin}_{u}  \|u-u_{\text{nom}}(x)\|^2\\
\text{s.t. }\ & \text{CBF constraint \eqref{eq:or_cbf} e.g., \eqref{eq:hs}, and \eqref{eq:ell-const}}.
\end{aligned}
\end{align}

This method prioritizes safety while minimizing deviation from the nominal controller. 
If infeasible, a modified optimization problem incorporating a penalty parameter $\rho > 0$ to balance safety and controller deviation can be used. 
In case of the ellipsoidal safe set, for example, the constraint \eqref{eq:ell-const} is revised:
\begin{align} \label{eq:ell_c_rev}
\begin{aligned} 
u ^*(x) &=  \text{argmin}_{u}  \|u-u_{\text{nom}}(x)\|^2 +\rho \delta\\
\text{s.t. }\ & 
 \bar{u}_t^\top \tilde{P}\bar{u}_t + 2\tilde{q}^\top(x) \bar{u}_t+ \tilde{r}(x)\leq \delta \text{ and } 
\tilde{A}\bar{u}_t \leq 0.
\end{aligned}
\end{align}

%------------------
\subsection{Method 2: CLF-CBF-based optimization}
Another approach is to combine Control Lyapunov Functions (CLF) and CBF to obtain the control inputs \cite{AmeP13,AmeXG17}. 

Let us first define the CLF:
\begin{defi}[Control Lyapunov Function \cite{AgrS17}]
 A map $V :  \mathbb{R}^n \rightarrow \mathbb{R}$ is an exponential control Lyapunov function for the 
system $x_{t+1} = f(x_t, u_t)$ if there exists:
\begin{itemize}
\item positive constants $c_1$ and $c_2$ such that
$c_1\|x_t\|^2 \leq V(x_t)\leq c_2\|x_t\|^2$, and
\item  a control input $u_t: \mathbb{R}^m \rightarrow \mathbb{R}$, $\forall x_t \in \mathbb{R}^n$ and $c_3>0$ such
that $V(x_{t+1})-V(x_t) +c_3 \|x_t\|^2 \leq 0$.
\end{itemize}
\end{defi}
We consider aiming at stabilizing the estimated mean states $\bar{x}_{t|t}$, $\bar{x}_{t+1|t}$ by choosing an appropriate $V$ and forcing the condition $V(\bar{x}_{t+1|t})-V(\bar{x}_{t|t}) +c_3 \|\bar{x}_{t|t}\|^2 \leq 0$. 

Choose 
\begin{align}
V(x_t) = x_t^\top \Phi x_t
\end{align}
for some $\Phi \in \mathbb{S}^n_+$.
Choosing $c_1 = \sigma_{\min}$ and $c_2 = \sigma_{\max}$, where $\sigma_{\min}$ and $\sigma_{\max}$ are the smallest and largest singular values of $\Phi$, respectively, satisfies the first condition of CLF. 

For the second condition to be satisfied, $u_t$ must exist such that satisfies the quadratic constraint.
 \begin{align} \begin{aligned}
&V(\bar{x}_{t+1|t})-V(\bar{x}_{t|t}) +c_3 \|\bar{x}_{t|t}\|^2\\
 = &(A\bar{x}_{t|t}+Bu_t)^\top \Phi (A\bar{x}_{t|t}+Bu_t)-\bar{x}_{t|t}^\top \Phi \bar{x}_{t|t} +c_3 \|\bar{x}_{t|t}\|^2 \\
 = &u_t ^\top B^\top \Phi Bu_t  +2(A\bar{x}_{t|t})^\top \Phi Bu_t \\
 &\qquad+ \bar{x}_{t|t}^\top (A^\top \Phi A -\Phi +c_3 I) \bar{x}_{t|t} \leq 0. \label{eq:clf_constraint}
\end{aligned}\end{align}

Combining this with the CBF constraint,  at each time, the control input can be computed by
\begin{align} \label{eq:CLF-CBF}
\begin{aligned}
 \mathbf{v} ^*(x) &= \argmin_{\mathbf{v}= [u_t, \delta]^\top \in\mathbb{R}^{m+1}} \frac{1}{2}\mathbf{v}^\top \Theta \mathbf{v} + \eta^\top\mathbf{v}\\
\text{s.t. }\ & 
u_t ^\top B^\top \Phi Bu_t  +2(A\bar{x}_{t|t})^\top \Phi Bu_t \\
 &\qquad+ \bar{x}_{t|t}^\top (A^\top \Phi A -\Phi +c_3 I) \bar{x}_{t|t} \leq  \delta,\\
 & \text{ CBF constraint \eqref{eq:or_cbf} e.g., \eqref{eq:hs}, and \eqref{eq:ell-const}}
\end{aligned}\end{align}
where a positive definite matrix $ \Theta \in \mathbb{S}^{m+1}_+$ and a positive vector $\eta  \in  \mathbb{R}^{m+1}$ are weights and 
$\delta$ is a relaxation variable that ensures the solvability of the optimization problem by relaxing the constraint \eqref{eq:clf_constraint}.

%%%%%%%%%%%%%%%%%%%%%%%%%%%%%%%%%%%%%%%%%%%%%%%%%%%%%%%%%%%%%%%%%%%%%%%%%%%%%%%%
\section{Numerical Examples}\label{sec:ex}
In this section, the effectiveness of the control strategies developed in Section \ref{sec:syn} is demonstrated through numerical simulations, using an example of vehicle navigation. 

Consider the state vector $x_t = \left[\begin{matrix} x_{t,1} & x_{t,2} \end{matrix}\right]^\top$, representing the vehicle's position $x_{t,1}$ and velocity $x_{t,2}$ at time $t$. The control input $u_t$ denotes the commanded acceleration, while the output $z_t$ denotes the measured position at time $t$. With the sampling time $0.05$, the system's dynamics are modeled as follows:
\begin{align}
\begin{aligned}
x_{t+1} &=\left[\begin{matrix} 1 & 0.05\\ 0 & 1 \end{matrix}\right]x_t + \left[\begin{matrix} 0.0125 \\ 0.05 \end{matrix}\right] u_t + w_t, \\ 
z_t &= \left[\begin{matrix} 1 & 0 \end{matrix}\right]x_t + v_t.
\end{aligned}
\end{align}
The disturbance and noise covariances are:
\begin{align}
\begin{aligned}
Q &= \left[\begin{matrix}7.66 \times 10^{-5} & 3.06 \times 10^{-3} \\ 3.06\times 10^{-3} & 1.23\times 10^{-1} \end{matrix}\right],\\
R &= 0.09.
\end{aligned}
\end{align}

Methods 1 and 2 are illustrated for a half-space safe set:
\begin{align}
\mathcal{C} = \{x\in \mathbb{R}^2: h(x) = \left[\begin{matrix} 0.4 & 0.4 \end{matrix}\right] x + 1 \geq 0\}.
\end{align}
The design parameters $\varepsilon = 0.3$ and $\alpha = 0.7$, and the initial conditions $\bar{x}_0 = \left[\begin{matrix} 7 & 0 \end{matrix}\right]^\top$ and $P_0 = Q$ are used.
The nominal controller used in Method 1 is
\begin{align}
u(x) = - 15x_1 - 5x_2,  \label{eq:nom_m1}
\end{align}
and the parameters used in Method 2 are 
\begin{align}\begin{aligned}
\Psi &= \left[\begin{matrix} 100 & 0 \\ 0 & 1 \end{matrix}\right]\!, 
\Theta = \left[\begin{matrix} 10 & 0 \\ 0 & 0.1 \end{matrix}\right]\!, 
\eta = \left[\begin{matrix} 0 & 100 \end{matrix}\right]^\top\!\!\!\!, 
c_3 = 10.
\end{aligned}\end{align}

For the duration of time 4, the performances of the proposed controllers (Proposed) are compared with the performances of the controllers that
\begin{itemize}
\item disregard safety constraints assuming no disturbances (Ignore constraint ($w=0$)), and
\item  use the expected value instead of the worst-case CVaR (Expected value-based).
\end{itemize}
These comparisons are illustrated in Figs. \ref{fig:1} and \ref{fig:2} for Methods 1 and 2, respectively.
The trajectories of the Kalman filter's estimated state mean discussed in \ref{sec:filter} are also included.
 
We observe in both figures that ignoring safety constraints leads to trajectories entering unsafe regions even without disturbances. 
This underscores the importance of employing CBFs to maintain safety.
In Fig. \ref{fig:1}, this is the trajectory resulting from using the controller \eqref{eq:nom_m1}. 

The Kalman filter's estimation, which are used to obtain the control inputs, align closely with true dynamics for both the proposed and expected value-based controllers. 

On the other hand, the most part of the trajectories of expected value-based controller enters the unsafe region while resulting that of the proposed controller remain in the safe region regardless of the stochastic disturbance. This shows the use of expected value is not sufficient to remain in the safe region.
Moreover, in these cases, we see the trajectories also successfully approach to the origin as desired.

%\begin{figure}[h]
%\centering
%\begin{minipage}[b]{0.49\columnwidth}
%    \centering
%    \includegraphics[width=\linewidth, viewport=1 4 418 308, clip]{figures/fig1.pdf}  
%    \subcaption{Method 1}\label{fig:1}
%\end{minipage}
%\begin{minipage}[b]{0.49\columnwidth}
%    \centering
%   \includegraphics[width=\linewidth, viewport=1 4 418 308, clip]{figures/fig2.pdf} 
%    \subcaption{Method 2}\label{fig:2}
%\end{minipage}
%  \caption{Phase portrait for vehicle navigation: the red regions indicate the unsafe region $h(x)<0$}\label{fig:pp}
%\end{figure}

\begin{figure}[h]
\centering
\begin{minipage}[b]{0.61\columnwidth}
    \centering
    \includegraphics[width=\linewidth, viewport=1 4 418 308, clip]{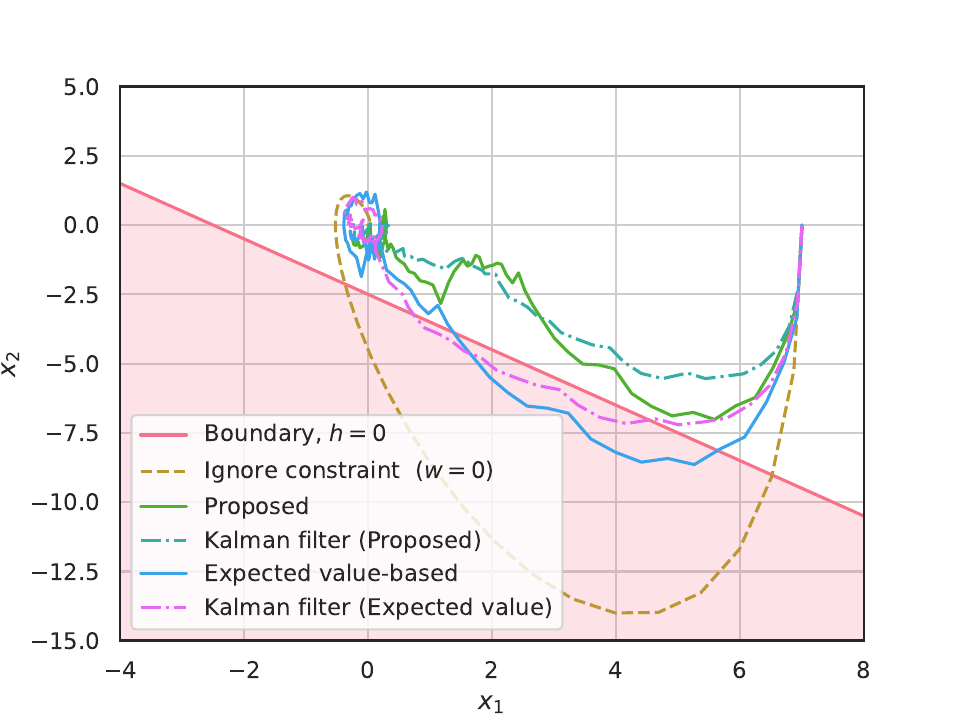}  \vspace{-.25in}
    \subcaption{Method 1}\label{fig:1}
\end{minipage}
\begin{minipage}[b]{0.61\columnwidth}
    \centering
   \includegraphics[width=\linewidth, viewport=1 4 418 308, clip]{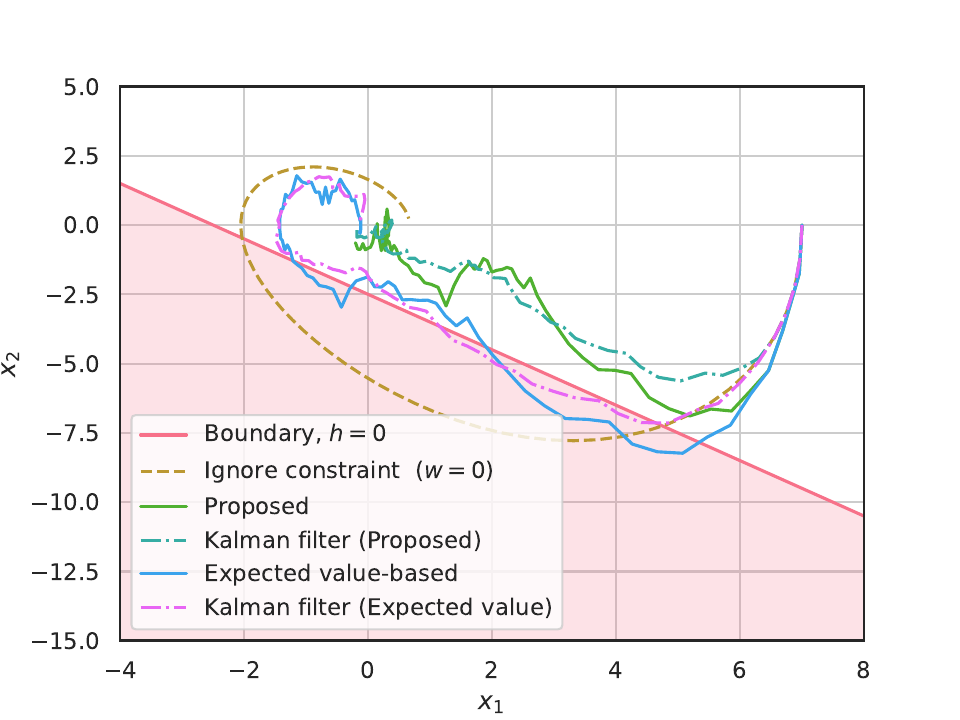}  \vspace{-.25in}
    \subcaption{Method 2}\label{fig:2}
\end{minipage}
  \caption{Phase portrait for vehicle navigation: the red regions indicate the unsafe region $h(x)<0$}\label{fig:pp} \vspace{-.24in}
\end{figure}

\section{Conclusions} \label{sec:conc}

This paper developed control approaches for discrete-time linear stochastic systems with partially observable states. Key contributions include:
\begin{itemize}
\item Employment of Kalman filter with the worst-case CVaR for effective incorporation of the tail risk into CBFs
\item Development of risk-aware constraints for the controller synthesis that effectively manage tail risks at the boundaries of safe sets
\item Integration of risk-aware constraints into two control methods: modification of nominal controllers and CLF-CBF-based optimization
\item Demonstration of improved safety and reliability through numerical examples of vehicle navigation
\end{itemize}

The proposed approaches contribute to the field of risk-aware control, offering a foundation for future research in enhancing safety measures.
Future research directions include investigating the applicability of other state observers within this framework to potentially broaden the approach's versatility, as well as expanding these techniques to nonlinear systems, thereby enhancing their scope and impact.

%%%%%%%%%%%%%%%%%%%%%%%%%%%%%%%%%%%%%%%%%%%%%%%%%%%%%%%%%%%%%%%%%%%%%%%%%%%%%%%%

\bibliographystyle{IEEEtran}
\bibliography{IEEEabrv,myref}

\end{document}